\documentclass[12pt]{amsart}

\usepackage{amsmath}
\usepackage{amssymb} 
\usepackage{stmaryrd}
\usepackage[all,cmtip]{xy}
\newcommand{\Z}{{\mathbb Z}} 
\newcommand{\N}{{\mathbb N}} 
\newcommand{\Q}{{\mathbb Q}}

\newcommand{\cH}{{\mathcal H}}
\newcommand{\cA}{{\mathcal A}}

\usepackage{enumerate}%

\DeclareMathOperator{\sign}{sign}

\newtheorem{theorem}{Theorem}

\newtheorem{prop}{Proposition}

\newtheorem*{rem*}{Remark}

\begin{document}  
\author{K.L Wong}
\address{ 
Department of Mathematics,
University of Hawaii, 
2565 McCarthy Mall, 
Honolulu, HI,  96822-2273 
}
\email{wongkl@math.hawaii.edu}

\title[]{Sums of Quadratic Functions with two Discriminants}
\begin{abstract} 
Zagier in \cite{Zagier} discusses a construction of a function $F_{k,D}(x)$  defined for an even integer $k \geq 2$, and a positive discriminant $D$. This construction is intimately related to 
half-integral weight modular forms. In particular, the average value of this function is a constant multiple of the $D$-th Fourier coefficient of weight $k+1/2$ Eisenstein series constructed by H. Cohen in \cite{Cohen}. 

In this note we consider a construction which works both for even and odd positive integers $k$. Our function $F_{k,D,d}(x)$
depends on two discriminants $d$ and $D$ with signs $\sign(d)=\sign(D)=(-1)^k$, degenerates to Zagier's function when $d=1$, namely,
\[
F_{k,D,1}(x)=F_{k,D}(x),
\]
and has very similar properties. In particular, we prove that the average value of $F_{k,D,d}(x)$ is again  a Fourier coefficient 
of H. Cohen's Eisenstein series of weight $k+1/2$, while now the integer $k \geq 2$ is allowed to be both even and odd.

\end{abstract}
\maketitle

\section{Introduction} \label{intro}

Let $\mathfrak{Q}_D$ be the set of all quadratic functions $Q=ax^2+bx+c=[a,b,c]$ with integer coefficients and of discriminant $D=b^2-4ac>0$. For an even positive integer $k \geq 2$, Zagier \cite{Zagier} defines the function $F_{k,D}: \mathbb{R} \to \mathbb{R}$ by
\begin{equation*}
\displaystyle F_{k,D}(x):=\sum_{\substack{Q \in \mathfrak{Q}_D \\a<0<Q(x)}}Q(x)^{k-1}
\end{equation*}
and investigates its striking properties. The construction raises an obvious question, what happens if $k$ is odd:
the function $ F_{k,D}(x)$ fails to have all these properties then.
In \cite[Section 9]{Zagier}, Zagier explains how one can gain the extra freedom and to allow $k$ to be odd:
he suggests to consider a symmetrization 
\[
 F_{k,\mathcal A}(x):=\sum_{\substack{Q \in \mathcal{A}\\a<0<Q(x)}}Q(x)^{k-1} + (-1)^k\sum_{\substack{Q \in \mathcal{-A}\\a<0<Q(x)}}Q(x)^{k-1}
\]
where the summation is restricted to quadratic forms in one equivalence class $\mathcal {A} \subset \mathfrak{Q}_D$ which is an
orbit in $\mathfrak{Q}_D$ under the action of $PSL_2(\Z)$, and
\[
-\mathcal{A} = \left\{-Q \ \vert \ Q \in \mathcal{A}\right\}.
\]
However, restricting to one class $\mathcal A$ does not allow for a generalization to odd $k$ of one of important properties of 
$F_{k,D}(x)$ which is discussed in \cite[Section 14]{Zagier}. Namely, one can define a constant 
$F_{k,0}$ such  that 
for every $x$, for even $k\geq 2$, the generating function $F_{k,0}+ \sum_D F_{k,D}(x) q^D$, where the sum is taken over all discriminants $D>0$, is the $q$-expansion of 
a modular form of weight $k+1/2$ in Kohnen's $+$-space. The functions $F_{k,D}(x)$ are $1$-periodic, and their average values are calculated by Zagier in  \cite[Section 8]{Zagier}. These are, up to a common multiple, $q$-expansion coefficients of H. Cohen's Eisenstein series. In order to state the result of this calculation, we 
denote by $\cH_k(\tau)$ the weight $k+1/2$ Eisenstein series on $\Gamma_0(4)$ introduced by H. Cohen in \cite{Cohen}:
\[
\cH_k(\tau)=\zeta(1-2k)+\sum_{(-1)^kD > 0} H(k,|D|) q^{|D|} \hspace{3mm} \textup{with $q=\exp(2 \pi i \tau)$ and $\Im(\tau)>0$ throughout.}
\]
The summation runs over discriminants $D$ such that $(-1)^kD > 0$, and $H(k,|D|)$ denote Cohen's numbers.
These are essentially the values at negative integers of Dirichlet $L$-function of the 
quadratic character associated with the field extension $\Q(\sqrt{D})/\Q$.
We refer to \cite{Cohen} for the definition of $H(k,D)$ and do not duplicate Cohen's definition in this paper.

The result of Zagier's calculation in \cite[Section 8]{Zagier} can now be stated as the identity
\begin{equation}  \label{zc}
\frac{\zeta(1-2k)}{\zeta(1-k) }\left(\frac{1}{2}\zeta(1-k)+\sum_{D>0} \int_0^1 F_{k,D}(x)   dx \ q^D\right)= \frac{1}{2} \cH_k(\tau)
\end{equation}
which holds for even $k \geq 2$.

In this paper, we present a generalization of $ F_{k,D}(x)$ which allows us to produce an exact analog of (\ref{zc}) for odd $k$.

Let $D$ be any discriminant, $d$ be a fundamental discriminant such that $\Delta := Dd>0$. For a quadratic form  $Q=ax^2+bx+c=[a,b,c]$ with integer coefficients and of discriminant
\[
b^2-4ac = \Delta,
\]
the value of genus character $\chi_d(Q)$ is defined (cf. \cite{Gross}) by
\[
\chi_d(Q) = 
\begin{cases} 
0 & \textup{if $(a,b,c,d)>1$ }\\
\left(\frac{d}{r}\right) & \textup{if $(a,b,c,d)=1,$ where $Q$ represents $r,\ \ $ $(r,d)=1$.}
\end{cases}
\]

We now assume that $k>1$ is an integer, and 
\[
\sign d = \sign D = (-1)^k.
\]
We define
\[
\displaystyle F_{k,D,d}(x):=\sum_{\substack{Q \in \mathfrak{Q}_{Dd} \\a<0<Q(x)}} \chi_d (Q) Q(x)^{k-1}.
\]

Note that our $F_{k,D,d}(x)$ generalizes Zagier's  $F_{k,D}(x)$ directly. Namely, for even $k>1$, we have
\[
F_{k,D,1}(x) = F_{k,D}(x).
\]

By the same argument as in \cite{Zagier}, our functions $F_{k,D,d}(x)$ are $1$-periodic and continuous for $k>1$, thus 
their average values make sense. The main result of this paper is the following generalization of (\ref{zc}).

\begin{theorem}
For  an integer $k>1$, and  a fundamental discriminant $d$ such that $\sign d = (-1)^k$,
\[
\frac{\zeta(1-2k)}{H(k,|d|)}\left(\frac{1}{2}H(k,|d|)+\sum_{(-1)^kD>0} \int_0^1 F_{k,D,d}(x)  dx \ q^{|D|}\right)= 
\frac{1}{2} \cH_k(\tau).
\]
\end{theorem}

It is quite natural to ask about the boundary case $k=1$. 
It follows from \cite{Jameson} that  $F_{1,D,d}(x)$ is defined if and only if $x$ is rational, so no averaging is possible.
At the same time, the series $\cH_1$ is not modular (see \cite{Cohen,Zn}). The following result checks with these observations.

\begin{theorem}
For a fundamental discriminant $d<0$ and a discriminant $D<0$ with $Dd$ being non-square, and $x \in \Q$, we have that
\[
F_{1,D,d}(x)=0.
\]
\end{theorem}

The proof of Theorem 1 is presented in Section 2.  
Equality of constant terms of $q$-series in Theorem 1 follows directly from the definition of Cohen's numbers $H(k,N)$
in \cite{Cohen}.
Thus Theorem 1 
is equivalent to the term-by-term identity
\begin{equation} \label{id}
\int_0^1 F_{k,D,d}(x)  dx  =  \frac{H(k,|D|)H(k,|d|)}{2\zeta(1-2k)},
\end{equation}
and that is what we prove in Section 2. This proof depends on two technical propositions
(Proposition \ref{mult} and  \ref{fact} in Section 2) which claim a decomposition of a certain
Dirichlet series into an Euler product, and calculate its Euler factors. The proofs of these propositions 
are presented in Section 3 of the paper. 


The value of genus character $\chi_d(Q)=\chi_d(\cA)$ depends only on the class $\cA \in \mathfrak{Q}_{Dd}$ such that $Q\in \cA$, not 
on the individual form $Q$ (see \cite{Gross} for details). It follows that 
\begin{equation} \label{sumid}
F_{k,D,d}(x) = \sum_{\cA} \chi_d(\cA) F^*_{k,\cA}(x),
\end{equation}
where the sum is taken over all classes $\cA$ of quadratic forms of discriminant $Dd$, and
\[
F^*_{k,\cA}(x) = \sum_{\substack{Q \in \cA \\ a<0<Q(x)}} Q(x)^{k-1}
\]
are introduced and briefly discussed in \cite[Section 9]{Zagier}. In particular, since $F^*_{k,\cA}(x)$ are periodic functions with period $1$, so are our $F_{k,D,d}(x)$, and the integrals in the left of (\ref{id}) may be interpreted as  average values of these functions.

In Section 4, we address the case when $k=1$. We show cancellations in (\ref{sumid}) which prove Theorem 2.
\vskip 10pt
{\sc Acknowledgement} 

The author is grateful to Prof. Pavel Guerzhoy for his advice and great support. His comments were very valuable to the writing of this paper.

\section{Proof of Theorem 1} \label{pt1}

In this section, we prove Theorem 1. 
\begin{proof}
All we need is to prove (\ref{id}). As in \cite[Section 8]{Zagier}, we have
\begin{align*}
\int_0^1 F_{k,D,d}(x) dx = \sum_{\substack{Q=[a,b,c]\in \mathfrak{Q}_{Dd}/\Gamma_{\infty} \\ a<0}} \chi_d (Q) \beta_k(Q), 
\end{align*}
where $\beta_k(Q):=\int_{-\infty}^{\infty} [\text{max}(0,Q(x))]^{k-1} dx$. 
We evaluate this integral using the substitution $x=\frac{-b+t\sqrt{Dd}}{2a}$:
\begin{align*}
\beta_k(Q)&= c_k(Dd)^{k-\frac{1}{2}}|a|^{-k}  \hspace*{5mm} \text{with} \hspace*{5mm} c_k:=\frac{1}{2^{2k-1}}\int_{-1}^1(1-t^2)^{k-1} dt = \frac{1}{2^{2k-1}}\frac{\Gamma(k) \Gamma (\frac{1}{2})}{\Gamma (k+\frac{1}{2})}.
\end{align*}

It follows that
\begin{align*}
\int_0^1 F_{k,D,d}(x) dx &= c_k |Dd|^{k-1/2}\sum_{\substack{Q=[a,b,c]\in \mathfrak{Q}_{Dd}/\Gamma_{\infty} \\ a<0}} \frac{\chi_d (Q)}{|a|^k}\\
&= c_k |Dd|^{k-1/2}\sum_{n=1}^{\infty} \left( \sum_{\substack{0 \leq b \leq 2n-1 \\ b^2 \equiv Dd \text{ mod } 4n }} \chi_d \left(\left[-n,b,\frac{Dd-b^2}{4n}\right]\right) \right) \frac{1}{n^k}.\\
\end{align*}


\begin{prop} \label{mult}
For a positive integer $n$, let
\[
 N_{D,d}(n):=\sum_{\substack{0 \leq b \leq 2n-1 \\ b^2 \equiv Dd \text{ mod } 4n }} \chi_d \left(\left[-n,b,\frac{Dd-b^2}{4n}\right]\right).
\]
The function $ (-1)^k N_{D,d} \colon \N \rightarrow \Z$ 
is multiplicative.
\end{prop}

We postpone a proof of Proposition \ref{mult} till Section 3, and continue with our proof of Theorem 1.


Proposition \ref{mult} allows us to write an Euler product expansion for the series
$\sum_{n=1}^{\infty} (-1)^kN_{D,d}(n)n^{-k}$, and we have that
\begin{align*}
\int_0^1 F_{k,D,d}(x) dx &= c_k |Dd|^{k-1/2}\sum_{n=1}^{\infty} \frac{N_{D,d}(n)}{n^k}\\
&= (-1)^kc_k |Dd|^{k-1/2}\sum_{n=1}^{\infty} \frac{(-1)^kN_{D,d}(n)}{n^k}\\
&= (-1)^kc_k |Dd|^{k-1/2} \prod_p \sum_{n=0}^{\infty} \frac{(-1)^kN_{D,d}(p^n)}{p^{nk}}.\\
\end{align*}

Our next proposition calculates the Euler factors in the above product
\begin{prop} \label{fact}
Let $p$ be a prime.
Let $D=D_0f^2$ with a fundamental discriminant $D_0$. Let $e\geq 0$ be the integer defined by $p^e || f$.
Then
\[
 \sum_{n=0}^{\infty} \frac{(-1)^kN_{D,d}(p^n)}{p^{nk}}= \frac{1-p^{-2k}}{\left(1-\left(\frac{D_0}{p}\right)p^{-k}\right)\left(1-\left(\frac{d}{p}\right)p^{-k}\right)}
 \frac{1}{(p^e)^{2k-1}} \left(\sigma_{2k-1}(p^{e})-\left(\frac{D_0}{p}\right)p^{k-1}\sigma_{2k-1}(p^{e-1})\right),
\]
where we adopt the usual convention $\sigma_{2k-1}(1/p)=0$.

\end{prop}

We postpone a proof of Proposition \ref{fact} till Section 3, and continue with our proof of Theorem 1.

Assume that $D=D_0f^2$ with a fundamental discriminant $D_0$, and let 
$f=\prod_{i=1}^{m} p_i^{e_i}$. 
An inductive argument  on the number of prime factors of $f$ allows us to conclude that
\[
\frac{1}{f^{2k-1}}\sum_{r|f} \mu (r)\left(\frac{D_0}{r}\right)r^{k-1}\sigma_{2k-1}\left(\frac{f}{r}\right) = \prod_{i=1}^{m} \frac{1}{(p_i^{e_i})^{2k-1}} \left(\sigma_{2k-1}(p_i^{e_i})+\mu(p_i)\left(\frac{D_0}{p_i}\right)p_i^{k-1}\sigma_{2k-1}(p_i^{e_i-1})\right).
\]

We take this equality into the account and use Proposition \ref{fact} to find that
\begin{align*}
&\int_0^1 F_{k,D,d}(x) dx \\
=& (-1)^k c_k |Dd|^{k-1/2}\prod_p \frac{1-p^{-2k}}{\left(1-\left(\frac{D_0}{p}\right)p^{-k}\right)\left(1-\left(\frac{d}{p}\right)p^{-k}\right)}\frac{1}{f^{2k-1}} \sum_{r|f} \mu (r)\left(\frac{D_0}{r}\right)r^{k-1}\sigma_{2k-1}\left(\frac{f}{r}\right)\\
=& (-1)^k c_k |Dd|^{k-1/2} L_{D_0}(k)L_{d}(k)\frac{1}{\zeta(2k)}\frac{1}{f^{2k-1}} \sum_{r|f} \mu (r)\left(\frac{D_0}{r}\right)r^{k-1}\sigma_{2k-1}\left(\frac{f}{r}\right)\\
=& (-1)^k c_k |D_0d|^{k-1/2} L_{D_0}(k)L_{d}(k)\frac{1}{\zeta(2k)}\sum_{r|f} \mu (r)\left(\frac{D_0}{r}\right)r^{k-1}\sigma_{2k-1}\left(\frac{f}{r}\right).
\end{align*}
Now standard functional equation for Dirichlet $L$-functions (and the definition of Cohen's numbers from \cite{Cohen}) 
allows us to derive
\begin{align*}
\int_0^1 F_{k,D,d}(x) dx  =\frac{H(k,|D|)H(k,|d|)}{2H(k,0)}
\end{align*}
which is equivalent to  (\ref{id}).
\end{proof}


\section{Proofs of Propositions \ref{mult} and \ref{fact}} \label{ppp12}

\begin{proof}[Proof of Proposition \ref{mult}]

Let $n_1$ and $n_2$ be two positive integers such that $(n_1,n_2)=1$. 
We want to prove that
\[
N_{D,d}(n_1n_2) = N_{D,d}(n_1) N_{D,d}(n_2).
\]

Without loss of generality, assume that $n_2$ is odd. Thus, 
$(n_2,4)=1$ and $(4n_1,n_2)=1$. 

We use our definition of $N_{D,d}(n)$ to transform these quantities. We obtain
\begin{align} \label{sum1}
\displaystyle N_{D,d}(n_1n_2)&=\sum_{\substack{0 \leq b \leq 2n_1n_2-1 \\ b^2 \equiv Dd \text{ mod } 4n_1n_2 }} \chi_d \left(\left[-n_1n_2,b,\frac{Dd-b^2}{4n_1n_2}\right]\right) \notag\\
&=\sum_{\substack{0 \leq b \leq 2n_1n_2-1 \\ b^2 \equiv Dd \text{ mod } 4n_1n_2 }} \chi_d \left(\left[-n_1,b,\frac{Dd-b^2}{4n_1n_2}\cdot n_2\right]\right)\chi_d \left(\left[n_2,b,\frac{Dd-b^2}{4n_1n_2}\cdot (-n_1)\right]\right) \notag\\
&=\sum_{\substack{0 \leq b \leq 2n_1n_2-1 \\ b^2 \equiv Dd \text{ mod } 4n_1n_2 }} \chi_d \left(\left[-n_1,b,\frac{Dd-b^2}{4n_1}\right]\right)\chi_d \left(\left[n_2,b,-\frac{Dd-b^2}{4n_2}\right]\right) \notag\\
&=(-1)^k\sum_{\substack{0 \leq b \leq 2n_1n_2-1 \\ b^2 \equiv Dd \text{ mod } 4n_1n_2 }} \chi_d \left(\left[-n_1,b,\frac{Dd-b^2}{4n_1}\right]\right)\chi_d \left(\left[-n_2,b,\frac{Dd-b^2}{4n_2}\right]\right)
\end{align}

Now consider
\begin{align} \label{sum2}
\displaystyle N_{D,d}(n_1) N_{D,d}(n_2)&=\sum_{\substack{0 \leq b_1 \leq 2n_1-1 \\ b_1^2 \equiv Dd \text{ mod } 4n_1}} \chi_d \left(\left[-n_1,b_1,\frac{Dd-b_1^2}{4n_1}\right]\right)\sum_{\substack{0 \leq b_2 \leq 2n_2-1 \\ b_2^2 \equiv Dd \text{ mod } 4n_2 }} \chi_d \left(\left[-n_2,b_2,\frac{Dd-b_2^2}{4n_2}\right]\right) \notag\\
&=\sum_{\substack{0 \leq b_1 \leq 2n_1-1 \\ b_1^2 \equiv Dd \text{ mod } 4n_1 \\ 0 \leq b_2 \leq 2n_2-1 \\ b_2^2 \equiv Dd \text{ mod } 4n_2 }} \chi_d \left(\left[-n_1,b_1,\frac{Dd-b_1^2}{4n_1}\right]\right)\chi_d \left(\left[-n_2,b_2,\frac{Dd-b_2^2}{4n_2}\right]\right).
\end{align}

Note that  the sums (\ref{sum1}) and (\ref{sum2}) have same amounts of summands. Indeed,
denote by $v(n)$ be the number of solutions of  $b^2-Dd \equiv 0 \text{ (mod } n)$. Then
the number of summands in (\ref{sum1}) is 
\[
\frac{1}{2}v(4n_1n_2)=\frac{1}{2}v(4n_1)v(n_2)
\] while the number of summands in 
 (\ref{sum2}) is 
\[
\frac{1}{2}v(4n_1)\cdot \frac{1}{2}v(4n_2)=\frac{1}{2}v(4n_1)\cdot \frac{1}{2}v(4)v(n_2)=\frac{1}{2}v(4n_1)v(n_2).
\] 

We now establish a one-to-one correspondence between these sets of summands such that corresponding summands are equal.

Summand in (\ref{sum2}) are numerated by pairs $(b_1,b_2)$ of residues modulo $2n_1$ and $2n_2$ correspondingly 
(which satisfy additional congruence conditions modulo $4n_1$ and $4n_2$.)
The Chinese Remainder Theorem allows us to find $B$ (unique modulo $4n_1n_2$) such that 
\[
B \equiv b_1 \mod 4n_1 \hspace{3mm} \text{and $B \equiv b_2 \mod n_2$}
\]
We now lift $B$ to an integer, which we also denote by $B$ such that $0 \leq B < 4n_1n_2$, and 
set 
\[
b= \begin{cases} 
B & \text{if $B< 2n_1n_2$} \\
4n_1n_2-B & \text{if $B \geq 2n_1n_2$} 
\end{cases}
\]
It is easy to see that the above procedure establishes a one-to-one correspondence between the sets of summands 
in (\ref{sum1}) and (\ref{sum2}), and we now want to check that corresponding summands are equal.

Since $b \equiv b_1\pmod {4n_1}$, we set $b= b_1+4n_1m = b_1+(2n_1)(2m)$ for some integer $m$ and find that  
\[
\chi_d \left(\left[-n_1,b_1,\frac{Dd-b_1^2}{4n_1}\right]\right)= \chi_d \left(\left[-n_1,b,\frac{Dd-b^2}{4n_1}\right]\right).
\] 
Since $b \equiv b_2 \pmod {n_2}$, we set $b=b_2+n_2m$ for some integer $m$. 
The congruence $b_2^2 \equiv Dd \pmod 4$ implies $b_2 \equiv Dd \pmod 2$. 
Similarly, $b^2 \equiv Dd \pmod 4$ implies $b \equiv Dd \pmod 2$ and $b \equiv b_2 \pmod 2$. 
Since $n_2$ is odd, $m$ must be even, $m=2m'$.
Thus, $b=b_2+n_2m=b_2+n_2(2m')=b_2+2n_2(m')$. Now we have 

\[
\chi_d \left(\left[-n_2,b_2,\frac{Dd-b_2^2}{4n_2}\right]\right)= \chi_d \left(\left[-n_2,b,\frac{Dd-b^2}{4n_2}\right]\right). 
\]

It follows that 
\[
N_{D,d}(n_1n_2)=(-1)^k N_{D,d}(n_1) N_{D,d}(n_2),
\]
therefore
\[
(-1)^kN_{D,d}(n_1n_2)= [(-1)^k N_{D,d}(n_1)] [(-1)^k N_{D,d}(n_2)]
\]
as required.

\end{proof}


We now turn to the proof of Proposition \ref{fact}.
This proof varies slightly depending on whether the involved quantities are or are not divisible by $p$.
Also, the case $p=2$ has to be considered separately. 
In particular, we say that we are in {\bf Case 1} if $p \nmid f$, and in {\bf Case 2} if $p | f$.
In each case, we consider the following sub-cases

\begin{enumerate}[{\bf(i)}]
\item $p \nmid d$, $p \nmid D_0$ 
\item $p \nmid d$, $p|D_0$ 
\item $p | d$, $p \nmid D_0$
\item $p| d$, $p|D_0$,
\end{enumerate}
and in every sub-case we will have part {\bf(a)} if $p$ is odd, and part {\bf(b)} for $p=2$.

For the sake of space and clarity, we present here  proofs only for {\bf Case 1(i)(a)} and {\bf Case 2(iii)(a)}. While the former is the simplest generic case, we will use the latter to illustrate the ideas involved in these proofs. In the remaining cases, one exploits same
set of ideas, specifically, one uses an explicit calculation of the quantities $N_{D,d}(p^n)$. 

\begin{proof}[Proof of Proposition \ref{fact} in \textbf{\bf Case 1(i)(a)}]

Recall the assumptions: $p \nmid f$, $p \nmid d$ and $p \nmid D_0$ with $p$ odd.

We need to prove the identity
\[
 \sum_{n=0}^{\infty} \frac{(-1)^kN_{D,d}(p^n)}{p^{nk}} =  \frac{1-p^{-2k}}{\left(1-\left(\frac{D_0}{p}\right)p^{-k}\right)\left(1-\left(\frac{d}{p}\right)p^{-k}\right)}.
\]

As long as $p \nmid d$, we can use an explicit formula for the genus character proved in \cite{Gross} to get  
\[
\chi_d([-p^n,b,c])=\left(\frac{d}{-p^n}\right)\left(\frac{1}{c}\right)=\left(\frac{d}{-p^n}\right).
\] 
We thus have that
\[
N_{D,d}(p^n)=\sum_{\substack{0 \leq b \leq 2n-1 \\ b^2 \equiv Dd \text{ mod } 4p^n }} \chi_d ([-p^n,b,\frac{Dd-b^2}{4p^n}])=\left(\frac{d}{-p^n}\right)\sum_{\substack{0 \leq b \leq 2n-1 \\ b^2 \equiv Dd \text{ mod } 4p^n }} 1.
\]
We make use of notation  (cf. \cite[Section 8]{Zagier})   
\[ 
N_{\Delta}(n)=\sum_{\substack{0 \leq b \leq 2n-1 \\ b^2 \equiv \Delta \text{ mod } 4n }} 1
\] 
to obtain 
\begin{align*}
\sum_{n=0}^{\infty} \frac{(-1)^kN_{D,d}(p^n)}{p^{nk}} &= \sum_{n=0}^{\infty} \frac{(-1)^k\left(\frac{d}{-p^n} \right) N_{Dd}(p^n)}{p^{nk}}\\
&=\sum_{n=0}^{\infty} \frac{(-1)^k \left(\frac{d}{-1} \right) \left(\frac{d}{p^n} \right) N_{Dd}(p^n)}{p^{nk}}\\
&=\sum_{n=0}^{\infty} \frac{\left(\frac{d}{p} \right)^n N_{Dd}(p^n)}{p^{nk}}.
\end{align*}
Recall that $v(n)$ denotes the number of solutions of  $b^2-Dd \equiv 0 \text{ (mod } n)$. 
Since $p$ is odd,
\[
N_{Dd}(p^n)= \frac{1}{2} \cdot  v(4p^n) \\
= \frac{1}{2} \cdot  v(4) \cdot  v(p^n) \\
= \frac{1}{2} \cdot  2 \cdot  v(p^n) \\
=v(p^n). \\
\]

If $\left(\frac{d}{p} \right) \not = \left(\frac{D}{p} \right)= \left(\frac{D_0}{p} \right)$, then $\left(\frac{Dd}{p} \right)=-1$ means that  
$Dd$ is a quadratic non-residue $\mod p$, therefore $v(p^n)=N_{Dd}(p^n) =0$ for $n \geq 1$, and
\[
 \sum_{n=0}^{\infty} \frac{(-1)^kN_{D,d}(p^n)}{p^{nk}}=1= \frac{(1+p^{-k})(1-p^{-k})}{\left(1-\left(\frac{D_0}{p} \right)p^{-k}\right)\left(1-\left(\frac{d}{p} \right)p^{-k}\right)}
\]
as required.

If $\left(\frac{d}{p} \right) = \left(\frac{D}{p} \right)$, then $\left(\frac{Dd}{p} \right)=1$ and $Dd$ is a quadratic residue modulo $p$.
Then Hensel's lemma implies that $v(p^n)=N_{Dd}(p^n) =2$ for $n \geq 1$, and we calculate
\begin{align*}
\sum_{n=0}^{\infty} \frac{(-1)^kN_{D,d}(p^n)}{p^{nk}} & =\sum_{n=0}^{\infty} \frac{\left(\frac{d}{p} \right)^n N_{Dd}(p^n)}{p^{nk}}\\
&=1+\frac{\left(\frac{d}{p} \right)\cdot 2}{p^k}+\frac{ \left(\frac{d}{p^2} \right) \cdot 2}{p^{2k}}+\frac{ \left(\frac{d}{p^3} \right)\cdot 2}{p^{3k}}+\cdots\\
&=1+\frac{2\left(\frac{d}{p} \right)p^{-k}}{1-\left(\frac{d}{p} \right)p^{-k}}\\
&=\frac{1+\left(\frac{d}{p} \right)p^{-k}}{1-\left(\frac{d}{p} \right)p^{-k}}\\
&=\frac{\left(1+\left(\frac{d}{p} \right)p^{-k}\right)\left(1-\left(\frac{d}{p} \right)p^{-k}\right)}{\left(1-\left(\frac{d}{p} \right)p^{-k}\right)\left(1-\left(\frac{d}{p} \right)p^{-k}\right)}\\
&=\frac{(1+p^{-k})(1-p^{-k})}{\left(1-\left(\frac{D_0}{p} \right)p^{-k}\right)\left(1-\left(\frac{d}{p} \right)p^{-k}\right)}.
\end{align*}
as required.
\end{proof}

\begin{proof}[Proof of Proposition \ref{fact} in \textbf{\bf Case 2(iii)(a)}]
Recall the assumptions: $p | f$, $p | d$ and $p \nmid D_0$ with $p$ odd. Furthermore, recall that integer $e >0$ is defined as the maximum power of $p$ dividing $f$, namely $p^e || f$.

Under these assumptions, one can calculate the quantities $N_{D,d}(p^n)$ to be:
\[
\begin{array}{lccc}
N_{D,d}(p^{2s-1})&= & 0  & \text{for}  \hspace{3mm}1 \leq s \leq e, \\
N_{D,d}(p^{2s})&=& (-1)^k(p^{s}-p^{s-1})  & \text{for} \hspace{3mm} 1 \leq s \leq e, \\
N_{D,d}(p^{2e+1})&=&(-1)^k\left(\frac{D_0}{p}\right)p^e ,\\
N_{D,d}(p^{n})&=&0 & \text{for} \hspace{3mm} n \geq 2e+2.
\end{array}
\]
Thus we have that

\begin{align*}
&\sum_{n=0}^{\infty} \frac{(-1)^kN_{D,d}(p^n)}{p^{nk}}
= 1+\frac{p-1}{p^{2k}}+\frac{p^2-p}{p^{4k}}+\frac{p^3-p^2}{p^{6k}}+\cdots+\frac{p^e-p^{e-1}}{p^{2ek}}+\frac{\left(\frac{D_0}{p}\right)p^e}{p^{(2e+1)k}}\\
=&1+\frac{1}{p^{2k-1}}+\frac{1}{p^{2(2k-1)}}+\cdots+\frac{1}{p^{e(2k-1)}}-\left(\frac{1}{p^{2k}}+\frac{1}{p^{4k-1}}+\frac{1}{p^{6k-2}}+\cdots+\frac{1}{p^{2ek-e+1}}\right)+\frac{\left(\frac{D_0}{p}\right)p^e}{p^{(2e+1)k}}\\
=&\frac{\sigma_{2k-1}(p^e)}{p^{e(2k-1)}}-\frac{p^{(e-1)(2k-1)}+p^{(e-2)(2k-1)}+\cdots+p^{2k-1}+1}{p^{e(2k-1)+1}}+\frac{\left(\frac{D_0}{p}\right)p^e}{p^{(2e+1)k}} \\
=&\frac{\sigma_{2k-1}(p^e)}{p^{e(2k-1)}}-\frac{\sigma_{2k-1}(p^{e-1})}{p^{e(2k-1)+1}}+\frac{\left(\frac{D_0}{p}\right)}{p^{e(2k-1)+k}}\\
=&\frac{p^k\sigma_{2k-1}(p^{e})+\left(\frac{D_0}{p}\right)\sigma_{2k-1}(p^{e})-\left(\frac{D_0}{p}\right)p^{2k-1}\sigma_{2k-1}(p^{e-1})-p^{k-1}\sigma_{2k-1}(p^{e-1})}{p^{e(2k-1)+k}}\\
=& \frac{p^k+\left(\frac{D_0}{p}\right)}{p^{k}} \cdot \frac{\sigma_{2k-1}(p^{e})-\left(\frac{D_0}{p}\right)p^{k-1}\sigma_{2k-1}(p^{e-1})}{p^{e(2k-1)}} \\
=& \left(1+\left(\frac{D_0}{p}\right)p^{-k}\right)\frac{1}{p^{e(2k-1)}} \left(\sigma_{2k-1}(p^{e})-\left(\frac{D_0}{p}\right)p^{k-1}\sigma_{2k-1}(p^{e-1})\right)\\
=& \frac{1-p^{-2k}}{\left(1-\left(\frac{D_0}{p}\right)p^{-k}\right)\left(1-\left(\frac{d}{p}\right)p^{-k}\right)} \frac{1}{(p^e)^{2k-1}} \left(\sigma_{2k-1}(p^{e})-\left(\frac{D_0}{p}\right)p^{k-1}\sigma_{2k-1}(p^{e-1})\right).
\end{align*}

\end{proof}
\section{Proof of Theorem 2} \label{pt2}

The statement follows easily from 
\begin{equation} \label{one}
F_{1,D,d}(x+1)=F_{1,D,d}(x),
\end{equation}

\begin{equation} \label{two}
F_{1,D,d}(0)=0
\end{equation}
and

\begin{equation} \label{three}
F_{1,D,d}\left(\frac{1}{x}\right)=F_{1,D,d}(x)
\end{equation}
for every $x \in \Q$.

It is easy to verify that (\ref{one}) holds. 

In order to check (\ref{two}), notice that

\begin{equation} \label{four}
\sum_{\substack{Q=[a,b,c]\in \mathbb{Z}^3 \\ b^2-4ac=Dd \\a<0<c}} \chi_d ([a,b,c])=
\sum_{\substack{Q=[a,b,c]\in \mathbb{Z}^3 \\ b^2-4ac=Dd \\c<0<a}} \chi_d ([a,b,c])=
0.
\end{equation}
since if $[a,b,c]$ appears in the sum, so does $[-c,b,-a]$ and $ \chi_d ([a,b,c])= -\chi_d ([-a,-b,-c])=-\chi_d ([-c,b,-a])$. 
Equation (\ref{two}) follows immediately, because the first sum equals $F_{1,D,d}(0)$.

We now prove (\ref{three}). We start with a transformation of $F_{1,D,d}(1/x)$:

\begin{align*}
F_{1,D,d}\left(\frac{1}{x}\right)=&\sum_{\substack{Q=[a,b,c]\in \mathbb{Z}^3 \\ b^2-4ac=Dd \\a<0\\Q\left(\frac{1}{x}\right)>0}} \chi_d ([a,b,c])\\
=& \sum_{\substack{Q=[a,b,c]\in \mathbb{Z}^3 \\ b^2-4ac=Dd \\a<0\\a\left(\frac{1}{x}\right)^2+b\left(\frac{1}{x}\right)+c>0}} \chi_d ([a,b,c])\\
=& \sum_{\substack{Q=[a,b,c]\in \mathbb{Z}^3 \\ b^2-4ac=Dd \\a<0\\a+bx+cx^2>0}} \chi_d ([a,b,c])\\
=& \sum_{\substack{Q=[c,b,a]\in \mathbb{Z}^3 \\ b^2-4ac=Dd \\c<0\\ax^2+bx+c>0}} \chi_d ([c,b,a]) \hspace*{1cm}\text{ (switched the names } a \text{ and } c)\\
=& \sum_{\substack{Q=[a,b,c]\in \mathbb{Z}^3 \\ b^2-4ac=Dd \\c<0\\ax^2+bx+c>0}} \chi_d ([a,b,c])  \hspace*{1cm}\text{ (by properties of } \chi_d).\\
\end{align*}

It follows that
\begin{align*}
&F_{1,D,d}\left(\frac{1}{x}\right)-F_{1,D,d}(x)\\
=& \sum_{\substack{Q=[a,b,c]\in \mathbb{Z}^3 \\ b^2-4ac=Dd \\c<0\\ax^2+bx+c>0}} \chi_d ([a,b,c])  - \sum_{\substack{Q=[a,b,c]\in \mathbb{Z}^3 \\ b^2-4ac=Dd \\a<0\\ax^2+bx+c>0}} \chi_d ([a,b,c]) \\
=& \sum_{\substack{Q=[a,b,c]\in \mathbb{Z}^3 \\ b^2-4ac=Dd \\c<0<a\\ax^2+bx+c>0}} \chi_d ([a,b,c])- \sum_{\substack{Q=[a,b,c]\in \mathbb{Z}^3 \\ b^2-4ac=Dd \\a<0<c\\ax^2+bx+c>0}} \chi_d ([a,b,c])  \hspace*{1cm}{(ac \not = 0 \text{ since } Dd \text{ is not a square})}\\
=& \sum_{\substack{Q=[a,b,c]\in \mathbb{Z}^3 \\ b^2-4ac=Dd \\c<0<a\\ax^2+bx+c>0}} \chi_d ([a,b,c])+ \sum_{\substack{Q=[a,b,c]\in \mathbb{Z}^3 \\ b^2-4ac=Dd \\-a>0>-c\\-ax^2-bx-c<0}} \chi_d ([-a,-b,-c])\\
=& \sum_{\substack{Q=[a,b,c]\in \mathbb{Z}^3 \\ b^2-4ac=Dd \\c<0<a\\ax^2+bx+c>0}} \chi_d ([a,b,c])+ \sum_{\substack{Q=[a,b,c]\in \mathbb{Z}^3 \\ b^2-4ac=Dd \\a>0>c\\ax^2+bx+c<0}} \chi_d ([a,b,c]) \hspace*{1cm}  
\begin{tabular}{l}
(\text{replaced } $-a,-b,-c$ \text{ by } $a,b,c$ \\
\text{ in the second sum})
\end{tabular}
\\
=& \sum_{\substack{Q=[a,b,c]\in \mathbb{Z}^3 \\ b^2-4ac=Dd \\c<0<a}} \chi_d ([a,b,c]) \hspace*{1cm} (ax^2+bx+c \not = 0 \text{ since } Dd \text{ is not a square})\\
=&  0 \text{ by (\ref{four})}.
\end{align*}

\qed


\end{document}